\documentclass[11pt,a4paper,reqno]{amsart}

\usepackage{amsfonts}
\usepackage{graphicx}
\usepackage{graphics}
\usepackage{epsfig}
\usepackage{anysize}
\usepackage{lipsum}
\usepackage{wrapfig}
\usepackage{coloring}

\usepackage{amsmath}
\usepackage{amssymb}
\usepackage{amsfonts}
\usepackage{amsthm,amscd}

\marginsize{3cm}{3cm}{3cm}{3cm}
\newtheorem{theo}{Theorem}[section]

\newtheorem{lem}{Lemma}[section]

\newtheorem{definition}{Definition}[section]

\makeatletter

\begin{document}
\setcounter{page}{1}
\vspace{2cm}
\author{Khaoula. Bouguetof $^1$, Nasser-eddine. Tatar$^2$}
\title[\centerline{ On nonexistence of global solutions for a semilinear  equation...\hspace{0.5cm}}]{On nonexistence of global solutions for a semilinear  equation with Hilfer- Hadamard fractional derivative}

\thanks{\noindent $^1$ Laboratory of Mathematics, Informatics and Systems, University of Larbi Tebessi, Tebessa 12002, Algeria\\
\noindent $^2$ Department of Mathematics and Statistics, King Fahd University of Petroleum and Minerals,
Dhahran 31261, Saudi Arabia \\
\indent \,\,\, e-mail:khaoula.bouguetof@univ-tebessa.dz, tatarn@kfupm.edu.sa\\
}

\begin{abstract}
For the following semilinear  equation with Hilfer- Hadamard fractional derivative
\begin{equation*}
\mathcal{D}^{\alpha_1,\beta}_{a^+} u-\Delta\mathcal{D}^{\alpha_2,\beta}_{a^+} u-\Delta u
=\vert u\vert^p, \hspace*{0.3cm} t>a>0,
\hspace*{0.3cm}  x\in\Omega,
\end{equation*}
where $\Omega\subset \mathbb{R}^N$ $(N\geqslant 1)$, $p>1$, $0<\alpha _{2}<\alpha _{1}<1$ and $0<\beta <1$. $\mathcal{D}^{\alpha_i,\beta}_{a^+}$ $(i=1,2)$ is the Hilfer- Hadamard fractional derivative of order $\alpha_i$ and of type $\beta$, we establish the 
necessary conditions for the existence of global solutions.

\end{abstract}
\maketitle 
\bigskip

  \section{Introduction}\label{sec:1}

\setcounter{section}{1}
\hspace*{0.3cm} We consider the following initial boundary value problem 
\begin{equation}\label{a}
\begin{cases}
\mathcal{D}^{\alpha_1,\beta}_{a^+} u-\Delta\mathcal{D}^{\alpha_2,\beta}_{a^+} u-\Delta u
=\vert u\vert^p, \hspace*{0.3cm}& t>a>0,
\hspace*{0.3cm}  x\in\Omega,\\
u(t,x)=0, \hspace*{0.3cm}& t>a>0,
\hspace*{0.3cm}  x\in\partial\Omega,\\
\bigg(\mathcal{D}^{(\beta-1)(1-\alpha_1)}_{a^+} u\bigg)(a,x)=u_0(x), \hspace*{0.3cm}& x\in \Omega,
\end{cases}
\end{equation}
where $\Omega $ is a bounded domain $\Omega\subset \mathbb{R}^N$ $(N\geqslant 1)$ with smooth boundary $\partial \Omega$, $p>1$, $0<\alpha _{2}<\alpha _{1}<1$, $0<\beta <1$ and $\Delta$ denotes the Laplacian operator with respect to the $x$ variable. The operator $ \mathcal{D}^{\alpha,\beta}_{a^+} $ is the Hilfer- Hadamard fractional derivative of order $\alpha$ and of type $\beta$, which we will be defined carefully  in a further part of this paper.
The equation\thinspace(\ref{a}) is a generalization of the well-known pseudo-parabolic equation of first order. The integer derivative is replaced by a fractional derivative in the sense of Hilfer-Hadamard. The second Hilfer-Hadamard fractional derivative of the Laplacian is allowed to be different from the first one.\\
\hspace*{0.3cm} Our objective is to find the range of $p$ for which nontrivial  solutions cannot exist for all time. This leads us to shed some light on the interaction of the nonlinear source term with $\Delta\mathcal{D}^{\alpha_2,\beta}_{a^+} u$. The analysis is based mainly on the test function method\thinspace\cite{mi}.\\
Let us first  recall some works related to our problem.\\
\hspace*{0.3cm} The semilinear pseudo-parabolic equation
\begin{equation}\label{j}
\begin{cases}
u_{t}-k\Delta u_{t}-\Delta u=|u|^{p},\hspace*{0.3cm}&(t,x)\in (0,\infty )\times 
\Omega,\\
u(t,x)=0, \hspace*{0.3cm}& (t,x)\in (0,\infty )\times\partial\Omega,\\
u(0,x)=u_0(x), \hspace*{0.3cm}& x\in \Omega,
\end{cases}
\end{equation}
 arises in many field of science and engineering : the aggregation of population\thinspace\cite{padra3n2004effect} and the nonstationary processes in semiconductors\thinspace\cite{korpusov2003three}. Eq.\thinspace(\ref{j}) is also called a Sobolev type equation, Sobolev Galpern type equation or the Benjamin Bona Mahony Burgers equation\thinspace\cite{al2011blow}. Many researchers have studied the existence and blow-up of solutions for problem\thinspace(\ref{j})\thinspace\cite{sun2019global, dangasymptotic}, by using different methods, such as the potential well method and the Galerkin method combined with the compactness method\thinspace. When $k=0$, Eq.\thinspace(\ref{j}) reduces to the heat equation 
 \begin{equation}\label{jj}
\begin{cases}
u_{t}-\Delta u=|u|^{p},\hspace*{0.3cm}&(t,x)\in (0,\infty )\times 
\Omega,\\
u(t,x)=0, \hspace*{0.3cm}& (t,x)\in (0,\infty )\times\partial\Omega,\\
u(0,x)=u_0(x), \hspace*{0.3cm}& x\in \Omega.
\end{cases}
\end{equation}
Fujita\thinspace\cite{fujita1966blowing} studied the global existence of mild solutions to\thinspace $(\ref{jj})$,  if $p>1+\frac{2}{N}$ and small initial data. In addition, he proved that the mild solution cannot  exist globally when $1<p< 1+\frac{2}{N}$ and  $ u_0\neq 0$.\\
\hspace*{0.3cm} In\thinspace\cite{weissler1981existence}, Weissler proved that if $p=1+\frac{2}{N}$ (critical case) and $u_0$ is small enough in $L^{q_{c}}(\mathbb{R}^N)$, $q_c=N(p-1)/2$, then the solution of\thinspace $(\ref{jj})$ exists globally.\\
\hspace*{0.3cm} Xu and Su\thinspace\cite{xu2013global}  showed that  all nontrivial solutions $u$ of the following problem 
\begin{equation}\label{jjj}
\begin{cases}
u_{t}-\Delta u_{t}-\Delta u=|u|^{p},\hspace*{0.3cm}&(t,x)\in (0,\infty )\times 
\Omega,\\
u(t,x)=0, \hspace*{0.3cm}& (t,x)\in (0,\infty )\times\partial\Omega,\\
u(0,x)=u_0(x), \hspace*{0.3cm}& x\in \Omega,
\end{cases}
\end{equation}
where $1<p<\infty$ if $N=1,2$; $1<p\leqslant\frac{N+2}{N-2}$ if $N\geqslant 3$, exist for all time under some conditions and they obtained sufficient conditions for nonexistence of solutions. In 2017, Xu and Zhou\thinspace\cite{xu2018lifespan} gave new blow-up and lifespan conditions of problem\thinspace(\ref{jjj}).\\ 
\hspace*{0.3cm}The rest of the paper is organized as follows: In Section\thinspace 2, we recall some definitions about Hilfer-Hadamard fractional integrals and derivatives. In Section\thinspace 3\thinspace, we study the
absence of global nontrivial weak solutions.
\section{Preliminary}
In this section,  we present some
results and basic properties of fractional calculus.  For more details, we refer to\thinspace\cite{kk, k, kkkk, kkk}.
\begin{definition}\label{yuu}  The Hadamard fractional integrals of order $\alpha>0$ of a function $\varphi\in L^q[a,b]$ $(1\leqslant q<\infty, 0< a\leqslant b\leqslant +\infty)$, are defined by
\begin{align*}
\big(\mathcal{I}^\alpha_{a^+}\varphi\big)(t)=\frac{1}{\Gamma(\alpha)}\int_a^t
\bigg(\log\frac{t}{\tau}\bigg)^{\alpha-1}\varphi(\tau)~\frac{d\tau}{\tau},\hspace*{0.3cm} a<t<b,
\end{align*}
and
\begin{align*}
\big(\mathcal{I}^\alpha_{b^-}\varphi\big)(t)=\frac{1}{\Gamma(\alpha)}\int_t^b
\bigg(\log\frac{\tau}{t}\bigg)^{\alpha-1}\varphi(\tau)~\frac{d\tau}{\tau}, \hspace*{0.3cm} a<t<b.
\end{align*}
\end{definition}
\begin{definition}\label{yu} Let $0<a<t<b$ and $n-1<\alpha<n$. The Hadamard fractional derivatives of order $\alpha$ for a function $\varphi$ are defined by
\begin{align*}
\big(\mathcal{D}^\alpha_{a^+}\varphi\big)(t)=\frac{1}{\Gamma(n-\alpha)}
\bigg(t\frac{d}{dt}\bigg)^n\int_a^t
\bigg(\log\frac{t}{\tau}\bigg)^{n-\alpha-1}\varphi(\tau)~\frac{d\tau}{\tau}= \delta^n\bigg( I^{n-\alpha}_{a^+} \varphi\bigg)(t),
\end{align*}
and
\begin{align*}
\big(\mathcal{D}^\alpha_{b^-}\varphi\big)(t)=\frac{(-1)^n}{\Gamma(n-\alpha)}
\bigg(t\frac{d}{dt}\bigg)^n\int_t^b\bigg(\log\frac{\tau}{t}\bigg)^{n-\alpha-1}\varphi(\tau)~\frac{d\tau}{\tau}=(-1)^n\delta^n\bigg( I^{n-\alpha}_{b^-} \varphi\bigg)(t),
\end{align*}
where $ \delta=t\frac{d}{dt}$, $n=[\alpha]+1$, $[\alpha]$ denotes the integer part of number $\alpha$.
\end{definition}
\begin{definition} Let $0<\alpha<1$ and $0\leqslant\beta\leqslant 1$. The Hilfer- Hadamard fractional derivative of order $\alpha$ and  type $\beta$  is defined by
\begin{align*}
\big(\mathcal{D}^{\alpha,\beta}_{a^+} \varphi\big)(t)=\bigg(\mathcal{I}^{\beta(1-\alpha)}_{a^+}\delta\mathcal{I}^{(1-\beta)(1-\alpha)}_{a^+}\varphi\bigg)(t),
\end{align*}
that is,
\begin{align}\label{h}
\big(\mathcal{D}^{\alpha,\beta}_{a^+} \varphi\big)(t)=\mathcal{I}^{\beta(1-\alpha)}_{a^+}\bigg(t\frac{d}{dt}\bigg)\bigg(\mathcal{I}^{(1-\beta)(1-\alpha)}_{a^+}\varphi\bigg)(t).
\end{align}
\end{definition}
\begin{definition}\label{loo} Let $[a,b]$ be a finite interval of the half-axis $\mathbb{R}^+$ and $0\leqslant \gamma<1$. We introduce the weighted spaces of continuous functions
\begin{align}
C_{\gamma, \log}[a,b]&=\bigg\lbrace \varphi: [a,b]\rightarrow \mathbb{R}: \bigg(\log\frac{t}{a}\bigg)^\gamma \varphi(t)\in C[a,b]\bigg\rbrace,\\
C_{1-\gamma,\log}^\gamma[a,b]&=\bigg\lbrace \varphi\in C_{1-\gamma,\log}[a,b]: \mathcal{D}^\gamma_{a^+}\varphi\in C_{1-\gamma,\log}[a,b]\bigg\rbrace ,
\end{align}
and
\begin{align}
C_{\delta,\gamma}^n[a,b]=\bigg\lbrace \varphi: [a,b]\rightarrow \mathbb{R}: \delta^k \varphi\in C[a,b],\; 0\leqslant k\leqslant n-1,\; \delta^n \varphi\in C_{\gamma, \log}[a,b]\bigg\rbrace,
\end{align}
 where $ \delta=t\frac{d}{dt}$ and $n\in\mathbb{N}$. In particular, when $n=0$ we define
\begin{align*}
C_{\delta,\gamma}^0[a,b]=C_{\gamma, \log}[a,b].
\end{align*}
\end{definition}
\begin{definition}\label{mp}The Banach space $X^p_c(a,b)$ $(1\leqslant p\leqslant\infty, c\in\mathbb{R})$ consists  of those real-valued Lebesgue measurable functions $\varphi: (a,b)\rightarrow \mathbb{R}$ such that 
\begin{align}
&\Vert \varphi\Vert_{X^p_c}= \bigg(\int_a^b \big\vert t^c \varphi(t)\big\vert^p\frac{dt}{t}\bigg)^{1/p}<\infty,\hspace*{0.3cm} p<\infty,\\
&\Vert \varphi\Vert_{X^\infty_c}=\text{ess}\sup_{a\leqslant t\leqslant b}\bigg\vert t^c \varphi(t)\bigg\vert<\infty.
\end{align}
When $c=1/p$, we see that $X^p_{1/p}(a,b)=L^p(a,b)$.
\end{definition}
\begin{lem}[\cite{kkk}]\label{w} Let $\alpha>0$, $1\leqslant p\leqslant \infty$ and $\frac{1}{p}+\frac{1}{q}=1$. If $\varphi\in L^p(a,b)$ and $\Psi\in X^q_{-1/p}(a,b)$, then
\begin{align*}
\int_a^b \varphi(t) (\mathcal{I}^\alpha_{a^+}\Psi)(t)\frac{dt}{t}=\int_a^b ( \mathcal{I}^\alpha_{b^-}\varphi)(t) \Psi(t) \frac{dt}{t}.
\end{align*}
\end{lem}
\begin{lem}[\cite{kkkk}]\label{ww} Assume $\varphi\in C^1_{\gamma,\log}[a,b]$, for $a<t<b$, $0<\gamma<1$ and $0<\alpha<1$. Then $ \mathcal{D}^\alpha_{a^+}$ exists on $(a,b]$ and $ \mathcal{D}^\alpha_{b^-}$ on $[a,b)$ and can be represented as
\begin{align*}
 \big(\mathcal{D}^\alpha_{a^+}\varphi\big)(t)&=\frac{\varphi(a)}{\Gamma(1-\alpha)}\bigg(\log\frac{t}{a} \bigg)^{-\alpha}+\frac{1}{\Gamma(1-\alpha)}\int_a^t \bigg(\log\frac{t}{\tau} \bigg)^{-\alpha}\varphi'(\tau) d\tau,\\
 \big(\mathcal{D}^\alpha_{b^-}\varphi\big)(t)&=\frac{\varphi(b)}{\Gamma(1-\alpha)}\bigg(\log\frac{b}{t} \bigg)^{-\alpha}-\frac{1}{\Gamma(1-\alpha)}\int_t^b \bigg(\log\frac{\tau}{t} \bigg)^{-\alpha}\varphi'(\tau) d\tau,
\end{align*}
respectively.
\end{lem}
\begin{lem}[\cite{kkk}]\label{www} Let $0\leqslant\gamma<1$ and $0<\alpha$. If $\gamma\leqslant \alpha$, then the operator $\mathcal{I}^\alpha_{a^+}$ is bounded from $C_{\gamma,\log}(a,b)$ into $C(a,b)$. In particular, it is bounded in $C_{\gamma,\log}(a,b).$
\end{lem}
\begin{lem}[ \cite{kkkk}]\label{wwww} Let $\alpha>0$, $\beta>0$ and $0<\gamma<1$. Assume $0<a<b<\infty$, then for $\varphi\in C_{\gamma,\log}(a,b)$, 
\begin{align}\label{oj}
\mathcal{I}^\alpha_{a^+} \mathcal{I}^\beta_{a^+}\varphi= \mathcal{I}^{\alpha+\beta}_{a^+}\varphi,
\end{align}
for each $t\in (a,b]$. In particular, if $\varphi\in C[a,b] $ the relation (\ref{oj}) is valid at any point $t\in[a,b]$.
\end{lem}

Throughout the next section, we take $0<\gamma=\alpha_1+\beta-\alpha_1\beta<1$.

\section{Blow-up of solutions}
\hspace*{0.3cm} First, we give the definition of weak solution of\thinspace$(\ref{a})$. After we prove the non-existence of nontrivial solutions.
\setcounter{section}{3}
\begin{definition}
\label{ab}Let $u_0\in C_0(\Omega)$ and $0<\alpha_2<\alpha_1<1.$ The function
\thinspace $u\in C^\gamma_{1-\gamma,\log}([a,b], C_0(\Omega))$\thinspace\ is a weak solution
of problem\thinspace $(\ref{a})$,  if
\begin{align}\label{weak}
 \int_{\Omega}\int_a^T \tilde{\varphi}\mathcal{D}^{\alpha_1,\beta}_{a^+} u dtdx-\int_{\Omega}\int_a^T \Delta\tilde{\varphi} \mathcal{D}^{\alpha_2,\beta}_{a^+} u dtdx-\int_{\Omega}\int_a^T \Delta\tilde{\varphi} udtdx= \int_{\Omega}\int_a^T \vert u\vert ^p \tilde{\varphi} dtdx,
 \end{align}

for all compactly supported test function\thinspace $\tilde{\varphi} \in C^{1,2 }_{t,x}([a,T]\times%
\Omega)$.
\end{definition}
\begin{theo}
Let $u_0\in C_0(\Omega)$ and $u_{0}\geqslant0.$ If  
\[1<p<\frac{\alpha_2N+1}{(\alpha_2 N+1-2\alpha_2)},\]
 then the problem\thinspace(\ref{a}) does not admit  global nontrivial solutions in the space $C^\gamma_{1-\gamma,\log}([a,b], C_0(\Omega))$.
\end{theo}

\begin{proof}
We assume the contrary. Let $\Phi\in C_0^\infty([0,\infty))$  be a decreasing function satisfying
\begin{align*}
\Phi(\sigma)=
\begin{cases}
1,\hspace*{0.5cm} 0\leqslant \sigma\leqslant 1,\\
0,\hspace*{0.5cm} \sigma\geqslant 2.
\end{cases}
\end{align*}
We define the function $\tilde{\varphi}(t,x)$ as follows
\begin{align}
\tilde{\varphi}(t,x)=\frac{\varphi_1(t)}{t}\varphi_2(x),
\end{align}
with $\varphi_1(t)\in C^1([a,\infty))$, $\varphi_1(t)\geqslant 0$ and $\varphi_1(t)$ is non-increasing such that
\begin{align}\label{yt}
\varphi_1(t)=
\begin{cases}
1,\hspace*{0.5cm} 0<a\leqslant t\leqslant \theta T,\hspace*{0.5cm} 0<\theta<1,\\
0,\hspace*{0.5cm} t\geqslant T,
\end{cases}
\end{align}
 for $ T> a>0$ and we choose 
\begin{align}
\varphi_2(x)=\bigg[\Phi\bigg(\frac{\Vert x\Vert}{T^{\alpha_2}}\bigg)\bigg]^\mu,\hspace*{0.5cm} \mu\geqslant  \frac{2p}{p-1}.
\end{align} 
 Equality\thinspace(\ref{weak}) actually reads 
 \begin{align}
 &\int_{\Omega_1}\int_a^T \varphi_2(x) \varphi_1(t) \mathcal{D}^{\alpha_1,\beta}_{a^+} u \frac{dt}{t}dx-\int_{\Omega_1}\int_a^T \Delta\varphi_2(x) \varphi_1(t) \mathcal{D}^{\alpha_2,\beta}_{a^+} u \frac{dt}{t}dx-\int_{\Omega_1}\int_a^T \Delta\varphi_2(x) \varphi_1(t) u\frac{dt}{t}dx\nonumber\\&= \int_{\Omega_1}\int_a^T \vert u\vert ^p \varphi_2(x) \varphi_1(t) \frac{dt}{t}dx.
 \end{align}
where $\Omega_1 :=\lbrace x\in\Omega:\hspace*{0.2cm}\Vert x\Vert\leqslant 2T^{\alpha_2}\rbrace$. From the definition of $\mathcal{D}^{\alpha,\beta}_{a^+} u$, we can re-write the above equation as
 \begin{align}
 &\int_{\Omega_1}\int_a^T \varphi_2(x) \varphi_1(t) \mathcal{I}^{\beta(1-\alpha_1)}_{a^+}\bigg(t\frac{d}{dt}\bigg)\bigg(\mathcal{I}^{(1-\beta)(1-\alpha_1)}_{a^+}u\bigg) \frac{dt}{t}dx\nonumber\\&-\int_{\Omega_1}\int_a^T \Delta\varphi_2(x) \varphi_1(t) \mathcal{I}^{\beta(1-\alpha_2)}_{a^+}\bigg(t\frac{d}{dt}\bigg)\bigg(\mathcal{I}^{(1-\beta)(1-\alpha_2)}_{a^+}u\bigg) \frac{dt}{t}dx-\int_{\Omega_1}\int_a^T \Delta\varphi_2(x) \varphi_1(t) u\frac{dt}{t}dx\nonumber\\&= \int_{\Omega_1}\int_a^T \vert u\vert ^p \varphi_2(x) \varphi_1(t) \frac{dt}{t}dx.
 \end{align}
 By Definition\thinspace\ref{loo}, we have $\bigg(\log\frac{t}{a}\bigg)^{1-\gamma} \mathcal{D}^\gamma _{a^+}u$ is continuous on $[a,T]$ implies that
 \begin{align*}
 \bigg\vert\bigg(\log\frac{t}{a}\bigg)^{1-\gamma} \mathcal{D}^\gamma _{a^+}u\bigg\vert\leqslant M, \hspace*{0.3cm}\forall t\in [a, T],
 \end{align*}
 for some positive constant $M$ (the constant $M$ will be a generic constant which may change at different places). Therefore
 \begin{align*}
 \int_a^T \bigg\vert t^{-1/p}\bigg( \mathcal{D}^\gamma_{a^+} u\bigg)(t)\bigg\vert^{p'}\frac{dt}{t}&\leqslant M^{p'} \int_a^T  t^{1-p'}\bigg(\log\frac{t}{a}\bigg)^{-p'(1-\gamma)}\frac{dt}{t}\nonumber\\
 &\leqslant M^{p'} \int_a^\infty  t^{1-p'}\bigg(\log\frac{t}{a}\bigg)^{-p'(1-\gamma)}\frac{dt}{t},
 \end{align*}
where $\frac{1}{p}+\frac{1}{p'}=1$. We introduce the following scaled variable
\begin{align*}
w=(p'-1)\log\big(t/a\big).
\end{align*}
Then
\begin{align}\label{n}
\int_a^T \bigg\vert t^{-1/p}\bigg(\mathcal{D}^\gamma_{a^+} u\bigg)(t)\bigg\vert^{p'}\frac{dt}{t}&\leqslant \frac{M^{p'} a^{1-p'}}{(p'-1)^{1-p'(1-\gamma)}}\int_0^\infty w^{-p'(1-\gamma)}\rm{e}^{-w}dw\nonumber\\
&\leqslant \frac{M^{p'} a^{1-p'}}{(p'-1)^{1-p'(1-\gamma)}} \Gamma(1-p'(1-\gamma))<\infty.
\end{align}
Consequently, $\big(\mathcal{D}^\gamma_{a^+} u\big)(t)\in X^{p'}_{-1/p}$. Thus it follows from Lemma\thinspace\ref{w} that
 \begin{align}\label{bp}
 &\int_{\Omega_1}\int_a^T \varphi_2(x)\mathcal{I}^{\beta(1-\alpha_1)}_{T^-} \varphi_1(t)\frac{d}{dt}\mathcal{I}^{(1-\beta)(1-\alpha_1)}_{a^+}u dtdx\nonumber\\&-\int_{\Omega_1}\int_a^T \Delta\varphi_2(x) \mathcal{I}^{\beta(1-\alpha_2)}_{T^-} \varphi_1(t)\frac{d}{dt}\mathcal{I}^{(1-\beta)(1-\alpha_2)}_{a^+}u dtdx-\int_{\Omega_1}\int_a^T \Delta\varphi_2(x) \varphi_1(t) u\frac{dt}{t}dx\nonumber\\&= \int_{\Omega_1}\int_a^T \vert u\vert ^p \varphi_2(x) \varphi_1(t) \frac{dt}{t}dx.
 \end{align}
  Using integration by parts  in (\ref{bp}), we obtain
  \begin{align}\label{pu}
& \int_{\Omega_1} \varphi_2(x)\bigg[(\mathcal{I}^{\beta(1-\alpha_1)}_{T^-} \varphi_1)(t)(\mathcal{I}^{(1-\beta)(1-\alpha_1)}_{a^+}u)(t,x)\bigg]_{t=a}^Tdx\nonumber\\&-\int_{\Omega_1}\int_a^T \varphi_2(x)\frac{d}{dt}\mathcal{I}^{\beta(1-\alpha_1)}_{T^-} \varphi_1(t)\mathcal{I}^{(1-\beta)(1-\alpha_1)}_{a^+}u dtdx\nonumber\\&
-\int_{\Omega_1}\Delta \varphi_2(x)\bigg[(\mathcal{I}^{\beta(1-\alpha_2)}_{T^-} \varphi_1)(t)(\mathcal{I}^{(1-\beta)(1-\alpha_2)}_{a^+}u)(t,x)\bigg]_{t=a}^T\nonumber\\&+\int_{\Omega_1}\int_a^T \Delta\varphi_2(x)\frac{d}{dt}\mathcal{I}^{\beta(1-\alpha_2)}_{T^-} \varphi_1(t)\mathcal{I}^{(1-\beta)(1-\alpha_2)}_{a^+}u dtdx-\int_{\Omega_1}\int_a^T \Delta\varphi_2(x) \varphi_1(t) u\frac{dt}{t}dx\nonumber\\&= \int_{\Omega_1}\int_a^T \vert u\vert ^p \varphi_2(x) \varphi_1(t) \frac{dt}{t}dx.
  \end{align}
  Since $\varphi_1\in C^1[a,b]$, then there exists a constant $M>0$ such that 
 $
 \vert \varphi_1(t)\vert \leqslant M.
$
 Hence
 \begin{align*}
 \vert \mathcal{I}^{\beta(1-\alpha_i)}_{T^-} \varphi_1(t)\vert &\leqslant \frac{M}{\Gamma(\beta(1-\alpha_i))}\int_t^T \bigg(\log\frac{\tau}{t}\bigg)^{\beta(1-\alpha_i)-1} \frac{d\tau}{\tau}\\
 &\leqslant \frac{M}{\Gamma(\beta(1-\alpha_i)+1)} \bigg(\log\frac{T}{t}\bigg)^{\beta(1-\alpha_i)},
 \end{align*}
where $i=1,2$. We see that $(\mathcal{I}^{\beta(1-\alpha_i)}_{T^-} \varphi_1)(T)=\lim _{t\longrightarrow T}\mathcal{I}^{\beta(1-\alpha_i)}_{T^-} \varphi_1(t)=0$ and 
  \begin{align}\label{bi}
 \bigg(\mathcal{I}^{(1-\beta)(1-\alpha_1)}_{a^+}u\bigg)(a^+,x)=\bigg(\mathcal{D}^{(\beta-1)(1-\alpha_1)}_{a^+} u\bigg)(a^+,x)=u_0(x).
 \end{align}
 It appears from Lemma\thinspace\ref{www} that
 \begin{align*}
 \bigg\vert \bigg(\log \frac{t}{a}\bigg)^{1-\gamma} u(.,x)\bigg\vert \leqslant M,
 \end{align*}
for $1-\gamma< (1-\beta)(1-\alpha_2)$. We can deduce
 \begin{align}\label{jy}
\bigg \vert\mathcal{I}^{(1-\beta)(1-\alpha_2)}_{a^+}u\bigg\vert&\leqslant\frac{1}{\Gamma((1-\beta)(1-\alpha_2))}\int_a^t \bigg\vert\bigg(\log\frac{t}{\tau}\bigg)^{(1-\beta)(1-\alpha_2)-1} u\frac{d\tau}{\tau}\bigg\vert\nonumber\\
 &\leqslant M \bigg(\log\frac{t}{a}\bigg)^{\gamma-1}\int_a^t \bigg(\log\frac{t}{\tau}\bigg)^{(1-\beta)(1-\alpha_2)-1} \frac{d\tau}{\tau}\nonumber\\
 &\leqslant M \bigg(\log\frac{t}{a}\bigg)^{\gamma-1+
 (1-\beta)(1-\alpha_2)}.
 \end{align}
 Therefore 
 \begin{align}\label{biii}
\bigg( \mathcal{I}^{(1-\beta)(1-\alpha_2)}_{a^+}u\bigg)(a,x)=0.
 \end{align}
Taking into account the above relations\thinspace(\ref{bi}) and (\ref{biii}) in (\ref{pu}), we find 
 \begin{align}\label{b}
 &-\int_{\Omega_1}\int_a^T \varphi_2(x)\frac{d}{dt}\mathcal{I}^{\beta(1-\alpha_1)}_{T^-} \varphi_1(t)\mathcal{I}^{(1-\beta)(1-\alpha_1)}_{a^+}u dtdx\nonumber\\&+\int_{\Omega_1}\int_a^T \Delta\varphi_2(x)\frac{d}{dt}\mathcal{I}^{\beta(1-\alpha_2)}_{T^-} \varphi_1(t)\mathcal{I}^{(1-\beta)(1-\alpha_2)}_{a^+}u dtdx-\int_{\Omega_1}\int_a^T \Delta\varphi_2(x) \varphi_1(t) u\frac{dt}{t}dx\nonumber\\&= \int_{\Omega_1}\int_a^T \vert u\vert ^p \varphi_2(x) \varphi_1(t) \frac{dt}{t}dx+\int_{\Omega_1}\varphi_2(x)(\mathcal{I}^{\beta(1-\alpha_1)}_{T^-} \varphi_1)(a)u_0(x)dx.
 \end{align} 
  Let 
 \begin{align*}
 &\mathcal{A}_1=-\int_{\Omega_1}\int_a^T \varphi_2(x)\frac{d}{dt}\mathcal{I}^{\beta(1-\alpha_1)}_{T^-} \varphi_1(t)\mathcal{I}^{(1-\beta)(1-\alpha_1)}_{a^+}u dtdx,\\
 \end{align*}
 and
 \begin{align*}
 &\mathcal{A}_2=\int_{\Omega_1}\int_a^T \Delta\varphi_2(x)\frac{d}{dt}\mathcal{I}^{\beta(1-\alpha_2)}_{T^-} \varphi_1(t)\mathcal{I}^{(1-\beta)(1-\alpha_2)}_{a^+}u dtdx.
 \end{align*}
  Multiplying $\mathcal{A}_1$ and $\mathcal{A}_2$ by $t/t$, we see that 
  \begin{align}
 \mathcal{A}_1=\int_{\Omega_1}\int_a^T \varphi_2(x)\bigg(- t\frac{d}{dt}\bigg)\mathcal{I}^{\beta(1-\alpha_1)}_{T^-} \varphi_1(t)\mathcal{I}^{(1-\beta)(1-\alpha_1)}_{a^+}u \frac{dt}{t}dx,
 \end{align}
 and
  \begin{align}
 \mathcal{A}_2=-\int_{\Omega_1}\int_a^T \Delta\varphi_2(x)\bigg(- t\frac{d}{dt}\bigg)\mathcal{I}^{\beta(1-\alpha_2)}_{T^-} \varphi_1(t)\mathcal{I}^{(1-\beta)(1-\alpha_2)}_{a^+}u \frac{dt}{t}dx.
 \end{align}
 Definition\thinspace\ref{yu} allows us to write
  \begin{align}
 \mathcal{A}_1=\int_{\Omega_1}\int_a^T \varphi_2(x)\bigg(\mathcal{D}^{1-\beta(1-\alpha_1)}_{T^-}\varphi_1\bigg)(t)\mathcal{I}^{(1-\beta)(1-\alpha_1)}_{a^+}u \frac{dt}{t}dx,
  \end{align}
  and
    \begin{align}
 \mathcal{A}_2=-\int_{\Omega_1}\int_a^T \Delta\varphi_2(x)\bigg(\mathcal{D}^{1-\beta(1-\alpha_2)}_{T^-}\varphi_1\bigg)(t)\mathcal{I}^{(1-\beta)(1-\alpha_2)}_{a^+}u \frac{dt}{t}dx.
 \end{align}
In view of Lemma\thinspace\ref{ww} and $(\ref{yt})$, we get
 \begin{align}\label{plr}
 \bigg(\mathcal{D}^{1-\beta(1-\alpha_i)}_{T^-}\varphi_1\bigg)(t)&= \frac{-1}{\Gamma(\beta(1-\alpha_i))}\int_t^T \bigg( \log\frac{s}{t}\bigg)^{\beta(1-\alpha_i)-1}\varphi_1'(s)ds\nonumber\\
 &=- \bigg( \mathcal{I}^{\beta(1-\alpha_i)}_{T^-}\delta\varphi_1\bigg)(t), \hspace*{0.5cm} i=1,2.
 \end{align}
According to (\ref{plr}), we have
 \begin{align}
 &\mathcal{A}_1=-\int_{\Omega_1}\int_a^T \varphi_2(x)\bigg(\mathcal{I}^{\beta(1-\alpha_1)}_{T^-}\delta\varphi_1\bigg)(t)\mathcal{I}^{(1-\beta)(1-\alpha_1)}_{a^+}u \frac{dt}{t}dx,
 \end{align}
 and
 \begin{align}
 &\mathcal{A}_2=\int_{\Omega_1}\int_a^T \Delta\varphi_2(x)\bigg(\mathcal{I}^{\beta(1-\alpha_2)}_{T^-}\delta\varphi_1\bigg)(t)\mathcal{I}^{(1-\beta)(1-\alpha_2)}_{a^+}u \frac{dt}{t}dx.
 \end{align}
Note that $\delta \varphi_1\in L^p([a,T])$ and by the same arguments as in the proof of $\big(\mathcal{D}^\gamma_{a^+} u\big)(t)\in X^{p'}_{-1/p}$ we may show that $ \mathcal{I}^{(1-\beta)(1-\alpha_i)}_{a^+}u\in X^{p'}_{-1/p}$ since $ \mathcal{I}^{1-\gamma}_{a^+}u\in C_{1-\gamma,\log}[a,T]$.\\

  Therefor, we see that Lemma\thinspace\ref{w} is satisfied
  \begin{align}
 &\mathcal{A}_1=-\int_{\Omega_1}\int_a^T \varphi_2(x)\delta\varphi_1(t)\bigg(\mathcal{I}^{\beta(1-\alpha_1)}_{a^+}\mathcal{I}^{(1-\beta)(1-\alpha_1)}_{a^+}u\bigg)(t,x) \frac{dt}{t}dx,\\
 &\mathcal{A}_2=\int_{\Omega_1}\int_a^T \Delta\varphi_2(x)\delta\varphi_1(t)\bigg(\mathcal{I}^{\beta(1-\alpha_2)}_{a^+}\mathcal{I}^{(1-\beta)(1-\alpha_2)}_{a^+}u\bigg)(t,x)\frac{dt}{t}dx.
 \end{align}
 Lemma\thinspace\ref{wwww} yields 
 \begin{align}
 &\mathcal{A}_1=-\int_{\Omega_1}\int_a^T \varphi_2(x)\delta\varphi_1(t)\bigg(\mathcal{I}^{1-\alpha_1}_{a^+}u\bigg)(t,x) \frac{dt}{t}dx,\\
 &\mathcal{A}_2=\int_{\Omega_1}\int_a^T \Delta\varphi_2(x)\delta\varphi_1(t)\bigg(\mathcal{I}^{1-\alpha_2}_{a^+}u\bigg)(t,x)\frac{dt}{t}dx.
 \end{align}
 By Definition\thinspace\ref{yuu} and the property of $\varphi_1$, we have
 \begin{align}\label{u}
 \mathcal{A}_1&\leqslant\frac{1}{\Gamma(1-\alpha_1)}\int_{\Omega_1}\int_a^T \varphi_2(x)\vert\delta\varphi_1(t)\vert\int_a^t \bigg(\log\frac{t}{s}\bigg)^{-\alpha_1}\frac{\vert u(s,x)\vert}{s}ds \frac{dt}{t}dx\nonumber\\
 &\leqslant\frac{1}{\Gamma(1-\alpha_1)}\int_{\Omega_1}\int_{\theta T}^T \varphi_2(x)\frac{\vert\delta\varphi_1(t)\vert}{\varphi_1^{1/p}(t)}\int_a^t \bigg(\log\frac{t}{s}\bigg)^{-\alpha_1}\frac{\vert u(s,x)\vert\varphi_1^{1/p}(s)}{s}ds \frac{dt}{t}dx\nonumber\\
 &\leqslant \int_{\Omega_1}\int_{\theta T}^T \varphi_2(x)\frac{\vert\delta\varphi_1(t)\vert}{\varphi_1^{1/p}(t)}\bigg(\mathcal{I}^{1-\alpha_1}_{a^+} \vert u\vert \varphi_1^{1/p}\bigg)(t,x)\frac{dt}{t}dx.
 \end{align}
A similar analysis for $\mathcal{A}_2$, we get 
 \begin{align}\label{om}
 \mathcal{A}_2&\leqslant\frac{1}{\Gamma(1-\alpha_2)}\int_{\Delta\Omega_1}\int_{\theta T}^T\vert\Delta \varphi_2(x)\vert\frac{\vert\delta\varphi_1(t)\vert}{\varphi_1^{1/p}(t)}\int_a^t \bigg(\log\frac{t}{s}\bigg)^{-\alpha_2}\frac{\vert u(s,x)\vert\varphi_1^{1/p}(s)}{s}ds \frac{dt}{t}dx\nonumber\\
 &\leqslant \int_{\Delta\Omega_1}\int_{\theta T}^T\vert\Delta \varphi_2(x)\vert\frac{\vert\delta\varphi_1(t)\vert}{\varphi_1^{1/p}(t)}\bigg(\mathcal{I}^{1-\alpha_2}_{a^+} \vert u\vert \varphi_1^{1/p}\bigg)(t,x) \frac{dt}{t}dx,
 \end{align}
where $\Delta\Omega_1 :=\lbrace x\in\Omega:\hspace*{0.2cm}T^{\alpha_2}\leqslant \Vert x\Vert \leqslant 2T^{\alpha_2}\rbrace$. We obtain from $(\ref{b})$, $(\ref{u})$ and $(\ref{om})$ 
 \begin{align*}
 &\int_{\Omega_1}\int_a^T \vert u\vert ^p \varphi_2(x) \varphi_1(t) \frac{dt}{t}dx+\int_{\Omega_1}\varphi_2(x)(\mathcal{I}^{\beta(1-\alpha_1)}_{T^-} \varphi_1)(a)u_0(x)dx \\&\leqslant \int_{\Omega_1}\int_{\theta T}^T \varphi_2(x)\frac{\vert\delta\varphi_1(t)\vert}{\varphi_1^{1/p}(t)}\bigg(\mathcal{I}^{1-\alpha_1}_{a^+} \vert u\vert \varphi_1^{1/p}\bigg)(t,x)\frac{dt}{t}dx\\&+\int_{\Delta\Omega_1}\int_{\theta T}^T\vert\Delta \varphi_2(x)\vert\frac{\vert\delta\varphi_1(t)\vert}{\varphi_1^{1/p}(t)}\bigg(\mathcal{I}^{1-\alpha_2}_{a^+} \vert u\vert \varphi_1^{1/p}\bigg)(t,x) \frac{dt}{t}dx+\int_{\Delta\Omega_1}\int_a^T \vert\Delta\varphi_2(x) \vert\varphi_1(t) u\frac{dt}{t}dx.
 \end{align*}
 The condition $u_0\geqslant 0$ yields 
 \begin{align}\label{kph}
 &\int_{\Omega_1}\int_a^T \vert u\vert ^p \varphi_2(x) \varphi_1(t) \frac{dt}{t}dx\nonumber\\&\leqslant \int_{\Omega_1}\int_{\theta T}^T \varphi_2(x)\frac{\vert\delta\varphi_1(t)\vert}{\varphi_1^{1/p}(t)}\bigg(\mathcal{I}^{1-\alpha_1}_{a^+} \vert u\vert \varphi_1^{1/p}\bigg)(t,x)\frac{dt}{t}dx\nonumber\\&+\int_{\Delta\Omega_1}\int_{\theta T}^T\vert\Delta \varphi_2(x)\vert\frac{\vert\delta\varphi_1(t)\vert}{\varphi_1^{1/p}(t)}\bigg(\mathcal{I}^{1-\alpha_2}_{a^+} \vert u\vert \varphi_1^{1/p}\bigg)(t,x) \frac{dt}{t}dx+\int_{\Delta\Omega_1}\int_a^T\vert \Delta\varphi_2(x)\vert \varphi_1(t) u\frac{dt}{t}dx.
 \end{align}
 It is easy to prove that $\big\vert u \varphi_1^{1/p}\big\vert \in X^{p'}_{-1/p}$ since 
    $u(.,x)\in C_{1-\gamma,\log}[a,T]$. Thus, we can apply Lemma\thinspace\ref{w} to obtain 
 \begin{align}\label{l}
& \int_{\Omega_1}\int_a^T \vert u\vert ^p \varphi_2(x) \varphi_1(t) \frac{dt}{t}dx \nonumber\\&\leqslant \int_{\Omega_1}\int_{\theta T}^T \varphi_2(x)\bigg(\mathcal{I}^{1-\alpha_1}_{T^-}\frac{\vert\delta\varphi_1(t)\vert}{\varphi_1^{1/p}(t)}\bigg)(t) \vert u\vert \varphi_1^{1/p}\frac{dt}{t}dx\nonumber\\&+\int_{\Delta\Omega_1}\int_{\theta T}^T \vert\Delta \varphi_2(x)\vert\bigg(\mathcal{I}^{1-\alpha_2}_{T^-}\frac{\vert\delta\varphi_1(t)\vert}{\varphi_1^{1/p}(t)}\bigg)(t) \vert u\vert \varphi_1^{1/p} \frac{dt}{t}dx+\int_{\Delta\Omega_1}\int_a^T \vert\Delta \varphi_2(x)\vert \varphi_1(t) u\frac{dt}{t}dx.
 \end{align}
  Using Young inequality with parameters \thinspace$p$\thinspace and\thinspace$ p'=\frac{p}{p-1}$, we have
 \begin{align}\label{i}
 &\int_{\Omega_1}\int_{\theta T}^T \varphi_2(x)\bigg(\mathcal{I}^{1-\alpha_1}_{T^-}\frac{\vert\delta\varphi_1(t)\vert}{\varphi_1^{1/p}(t)}\bigg)(t) \vert u\vert \varphi_1^{1/p}\frac{dt}{t}dx\nonumber\\&\leqslant \frac{1}{6p}\int_{\Omega_1}\int_{\theta T}^T \vert u\vert^p \varphi_1(t)\varphi_2(x)\frac{dt}{t}dx+\frac{6^{p'-1}}{p'}\int_{\Omega_1}\int_{\theta T}^T \varphi_2(x)\bigg\vert \mathcal{I}^{1-\alpha_1}_{T^-}\frac{\vert\delta\varphi_1(t)\vert}{\varphi_1^{1/p}(t)}\bigg\vert^{p'}\frac{dt}{t}dx\nonumber\\
 &\leqslant \frac{1}{6p}\int_{\Omega_1}\int_{a}^T \vert u\vert^p \varphi_1(t)\varphi_2(x)\frac{dt}{t}dx+\frac{6^{p'-1}}{p'}\int_{\Omega_1}\int_{\theta T}^T \varphi_2(x)\bigg\vert \mathcal{I}^{1-\alpha_1}_{T^-}\frac{\vert\delta\varphi_1(t)\vert}{\varphi_1^{1/p}(t)}\bigg\vert^{p'}\frac{dt}{t}dx,
 \end{align}
 \begin{align}\label{ii}
& \int_{\Delta\Omega_1}\int_{\theta T}^T \vert\Delta \varphi_2(x)\vert\bigg(\mathcal{I}^{1-\alpha_2}_{T^-}\frac{\vert\delta\varphi_1(t)\vert}{\varphi_1^{1/p}(t)}\bigg)(t) \vert u\vert \varphi_1^{1/p}\frac{dt}{t}dx\nonumber\\&\leqslant \frac{1}{6p}\int_{\Omega_1}\int_{\theta T}^T \vert u\vert^p \varphi_1(t)\varphi_2(x)\frac{dt}{t}dx+\frac{6^{p'-1}}{p'}\int_{\Delta\Omega_1}\int_{\theta T}^T \varphi_2(x)^{-p'/p}\vert\Delta \varphi_2(x)\vert^{p'}\bigg\vert \mathcal{I}^{1-\alpha_2}_{T^-}\frac{\vert\delta\varphi_1(t)\vert}{\varphi_1^{1/p}(t)}\bigg\vert^{p'}\frac{dt}{t}dx\nonumber\\
 &\leqslant \frac{1}{6p}\int_{\Omega_1}\int_{a}^T \vert u\vert^p \varphi_1(t)\varphi_2(x)\frac{dt}{t}dx+\frac{6^{p'-1}}{p'}\int_{\Delta\Omega_1}\int_{\theta T}^T \varphi_2(x)^{-p'/p}\vert\Delta \varphi_2(x)\vert^{p'}\bigg\vert \mathcal{I}^{1-\alpha_2}_{T^-}\frac{\vert\delta\varphi_1(t)\vert}{\varphi_1^{1/p}(t)}\bigg\vert^{p'}\frac{dt}{t}dx,
 \end{align}
 and
 \begin{align}\label{iii}
 & \int_{\Delta\Omega_1}\int_{a}^T \vert \Delta\varphi_2(x)\vert\varphi_1(t) u\frac{dt}{t}dx\nonumber\\&\leqslant \frac{1}{6p}\int_{\Omega_1}\int_{a}^T \vert u\vert^p \varphi_1(t)\varphi_2(x)\frac{dt}{t}dx+\frac{6^{p'-1}}{p'}\int_{\Delta\Omega_1}\int_{a}^T \varphi_2(x)^{-p'/p}\vert\Delta \varphi_2(x)\vert^{p'}\varphi_1(t) \frac{dt}{t}dx.
 \end{align}
 Using inequalities $(\ref{l})$, $(\ref{i}),$ $(\ref{ii})$ and $(\ref{iii})$, we obtain the inequality
 \begin{align}\label{lo}
 &\bigg(1-\frac{1}{2p}\bigg)\int_{\Omega_1}\int_a^T \vert u\vert ^p \varphi_2(x) \varphi_1(t) \frac{dt}{t}dx\nonumber\\&\leqslant \frac{6^{p'-1}}{p'}\int_{\Omega_1}\int_{\theta T}^T \varphi_2(x)\bigg\vert \mathcal{I}^{1-\alpha_1}_{T^-}\frac{\vert\delta\varphi_1(t)\vert}{\varphi_1^{1/p}(t)}\bigg\vert^{p'}\frac{dt}{t}dx\nonumber\\&+\frac{6^{p'-1}}{p'}\int_{\Delta\Omega_1}\int_{\theta T}^T \varphi_2(x)^{-p'/p}\vert\Delta \varphi_2(x)\vert^{p'}\bigg\vert \mathcal{I}^{1-\alpha_2}_{T^-}\frac{\vert\delta\varphi_1(t)\vert}{\varphi_1^{1/p}(t)}\bigg\vert^{p'}\frac{dt}{t}dx\nonumber\\&+\frac{6^{p'-1}}{p'}\int_{\Delta\Omega_1}\int_{a}^T \varphi_2(x)^{-p'/p}\vert\Delta \varphi_2(x)\vert^{p'}\varphi_1(t) \frac{dt}{t}dx.
 \end{align}
 We introduce the following scaled variable
 \begin{align*}
 \tau=\frac{t}{T},\hspace*{0.2cm} T\gg 1.
 \end{align*}
 It appears that
 \begin{align*}
 \int_{\theta T}^T\bigg\vert \mathcal{I}^{1-\alpha_i}_{T^-}\frac{\vert\delta\varphi_1(t)\vert}{\varphi_1^{1/p}(t)}\bigg\vert^{p'}\frac{dt}{t}&=\frac{1}{\Gamma^{p'}(1-\alpha_i)}\int_{\theta T}^T\bigg( \int_t^T \bigg(\log\frac{s}{t}\bigg)^{-\alpha_i}\frac{\vert\delta\varphi_1(s)\vert}{\varphi_1^{1/p}(s)}\frac{ds}{s}\bigg)^{p'}\frac{dt}{t}\nonumber\\
 &=\frac{1}{\Gamma^{p'}(1-\alpha_i)}\int_{\theta }^1\bigg( \int_{\tau T}^T \bigg(\log\frac{s}{\tau T}\bigg)^{-\alpha_i}\frac{\vert\varphi'_1(s)\vert}{\varphi_1^{1/p}(s)}ds\bigg)^{p'}\frac{d\tau}{\tau},
 \end{align*}
for $i=1,2$. Another change of variable $r=\frac{s}{T}$ yields
 \begin{align}\label{kp}
\int_{\theta }^1\bigg( \int_{\tau T}^T \bigg(\log\frac{s}{\tau T}\bigg)^{-\alpha_i}\frac{\vert\varphi'_1(s)\vert}{\varphi_1^{1/p}(s)}ds\bigg)^{p'}\frac{d\tau}{\tau}=\int_{\theta }^1\bigg( \int_{\tau }^1 \bigg(\log\frac{r}{\tau}\bigg)^{-\alpha_i}\frac{\vert\varphi'_1(r)\vert}{\varphi_1^{1/p}(r)}dr\bigg)^{p'}\frac{d\tau}{\tau}.
 \end{align}
 Since $\varphi_1\in C^1[a,\infty)$, we assume without loss of generality that
 \begin{align*}
\int_{\tau }^1 \bigg(\log\frac{r}{\tau}\bigg)^{-\alpha_i}\frac{\vert\varphi'_1(r)\vert}{\varphi_1^{1/p}(r)}dr\leqslant M,
 \end{align*}
 for $M>0$. Then Eq\thinspace(\ref{kp}) becomes
 \begin{align*}
 \frac{1}{\Gamma^{p'}(1-\alpha_i)}\int_{\theta }^1\bigg( \int_{\tau }^1 \bigg(\log\frac{r}{\tau}\bigg)^{-\alpha_i}\frac{\vert\varphi'_1(r)\vert}{\varphi_1^{1/p}(r)}dr\bigg)^{p'}\frac{d\tau}{\tau} < C\int_{\theta }^1 d\tau .
 \end{align*}
Putting $\theta=1-{\rm e}^{-T}$ with $T>a>0$, we obtain
 \begin{align}\label{ll}
 \int_{\theta T}^T\bigg\vert \mathcal{I}^{1-\alpha_i}_{T^-}\frac{\vert\delta\varphi_1(t)\vert}{\varphi_1^{1/p}(t)}\bigg\vert^{p'}\frac{dt}{t}\leqslant C{\rm e}^{-T},
 \end{align}
for some positive $C$ independent of $T$.\\

Next, using the change of variable $y=\frac{\Vert x\Vert}{T^{\alpha_2}}$, we get
\begin{align}\label{ta}
\int_{\Delta\Omega_1}\varphi_2(x)^{\frac{-p'}{p}}\big\vert \Delta\varphi_2(x)\big\vert^{p'}dx&=T^{\alpha_2 N-2\alpha_2 p'}\int_{1\leqslant \Vert y\Vert\leqslant 2}\bigg[\Phi(\Vert y\Vert)\bigg]^{-\frac{\mu}{p-1}}\bigg\vert\Delta\bigg[\Phi(\Vert y\Vert)\bigg]^\mu\bigg\vert^{\frac{p}{p-1}}dy\nonumber\\
&\leqslant T^{\alpha_2 N-2\alpha_2 p'}\int_{1\leqslant \Vert y\Vert\leqslant 2}\bigg[\Phi(\Vert y\Vert)\bigg]^{\frac{\mu(p-1)-2p}{p-1}}dy\nonumber\\
&\leqslant T^{\alpha_2 N-2\alpha_2 p'}.
\end{align}
Combining $(\ref{lo})$, (\ref{ll}) and (\ref{ta}), we get
 \begin{align}\label{kl}
 \frac{1}{2p}\int_{\Omega_1}\int_a^T \vert u\vert ^p \varphi_2(x) \varphi_1(t) \frac{dt}{t}dx< C{\rm e}^{-T} T^{\alpha_2 N}+C T^{\alpha_2 N-2\alpha_2 p'}+CT^{\alpha_2 N-2\alpha_2 p'+1},
 \end{align}
when \(T\longrightarrow +\infty\), we obtain
  \begin{align*}
\lim_{T\longrightarrow +\infty}{\rm e}^{-T} T^{\alpha_2 N}=0,
 \end{align*}
 and
 \begin{align*}
 \lim_{T\longrightarrow +\infty} T^{\alpha_2 N-2\alpha_2 p'+1}=0.
 \end{align*}
Therefore
 \begin{align}
 \lim_{T\rightarrow\infty}\int_{\Omega_1}\int_a^T \vert u\vert^p\varphi_1(t)\varphi_2(x)\frac{dt}{t}dx=0.
 \end{align}
 This leads to a contradiction.
 \end{proof}

 \nocite{*}

 \it

 \noindent

\end{document}